\documentclass{amsart}

\usepackage{geometry,graphicx,amssymb,amsmath,amsbsy,eucal,amsfonts,mathrsfs,ams
cd,bm}
\usepackage[all]{xy}

\numberwithin{equation}{section}

\allowdisplaybreaks[4]

\newtheorem{theorem}{Theorem}[section]
\newtheorem{lemma}[theorem]{Lemma}
\newtheorem{question}[theorem]{Question}
\newtheorem{corollary}[theorem]{Corollary}

\theoremstyle{definition}
\newtheorem{definition}[theorem]{Definition}
\theoremstyle{remark}
\newtheorem{remark}[theorem]{Remark}
\theoremstyle{definition}
\newtheorem{example}[theorem]{Example}

\newcommand{\vanish}[1]{}

\newcommand{\A}{\mathcal{A}}
\newcommand{\al}{\alpha}
\newcommand{\be}{\beta}
\newcommand{\ga}{\gamma}
\newcommand{\reals}{\mathbb{R}}
\newcommand{\integers}{\mathbb{Z}}
\newcommand{\set}[1]{\left\lbrace#1\right\rbrace}

\usepackage{color}

\newcommand{\margincolor}{red}      
\definecolor{darkgreen}{rgb}{0,0.7,0}

\newcounter{margincounter}
\newcommand{\marginnum}{
\ifnum\value{margincounter}<10
\textcolor{\margincolor}{\begin{picture}(0,0)\put(2.2,2.4){\circle{9}}\end{picture}\footnotesize\arabic{margincounter}}
\else\ifnum\value{margincounter}<100
\textcolor{\margincolor}{\begin{picture}(0,0)\put(4.256,2.5){\circle{11}}\end{picture}\footnotesize\arabic{margincounter}}
\else
\textcolor{\margincolor}{\begin{picture}(0,0)\put(6.8,2.5){\circle{14}}\end{picture}\footnotesize\arabic{margincounter}}
\fi\fi
}

\newcommand{\newword}[1]{{\em{#1}}}

\subjclass[2010]{Primary }
\begin{document}
\title{Coxeter arrangements in three dimensions}

\author{Richard Ehrenborg}
\address{Department of Mathematics \\ University of Kentucky, Lexington, KY 40506}
\email{jrge@ms.uky.edu}
\author{Caroline Klivans}
\address{Division of Applied Mathematics and Department of Computer Science \\ Brown University, Providence, RI 02906}
\email{klivans@brown.edu}
\author{Nathan Reading}
\address{Department of Mathematics, North Caroline State University, Raleigh, NC 27695} 
\email{reading@math.ncsu.edu}

\date{\today}

 \thanks{}
 

\begin{abstract}
  Let $\A$ be a finite real linear hyperplane arrangement in three
  dimensions.  Suppose further that all the regions of $\A$ are
  isometric.  We prove that $\A$ is necessarily a Coxeter arrangement.
  As it is well known that the regions of a Coxeter arrangement are
  isometric, this characterizes three-dimensional Coxeter
  arrangements precisely as those arrangements with isometric regions.
  It is an open question whether this suffices to characterize
  Coxeter arrangements in higher dimensions.
  We also present the three families of affine arrangements in
  the plane which are not reflection arrangements, but
  in which all the regions are isometric.
\end{abstract}

\maketitle

\section{Introduction}

In this note, we are concerned with the polyhedral geometry of Coxeter arrangements (hyperplane arrangements associated to finite Coxeter groups).
Coxeter groups are defined by  certain presentations,
but finite Coxeter groups coincide with finite real reflection groups.
The geometry, topology and combinatorics of Coxeter
arrangements have been extensively studied.
See e.g.~\cite{BB, Davis, GB}.

The regions of a real hyperplane arrangement are the connected components of the complement of the arrangement.  
Previous work of the second author,
with Drton~\cite{DK} and with Swartz~\cite{KS},
investigated certain projection volumes associated to a
region of an arrangement.  It was shown that the average volumes, over
all regions, are given by the coefficients of the characteristic
polynomial of the arrangement~\cite{DK,KS}.
When all of the regions of the hyperplane arrangement are isometric, this result determines the precise projection volumes for each region.  

Coxeter arrangements are a class of examples with the property that all regions of the arrangement are isometric.
As further examples could not be found, the following question arose:

\begin{question}[\cite{KS} Problem 13] \label{the q}
Does there exist a real central hyperplane arrangement with all regions isometric that is not a Coxeter arrangement?  
\end{question}

The main result of this paper is a negative answer to the question in three dimensions.
\begin{theorem}\label{thm:main}
 Any real central hyperplane arrangement in $\mathbb{R}^3$
 in which all the regions are isometric is a Coxeter arrangement. 
\end{theorem}

The proof of Theorem~\ref{thm:main} proceeds in several steps.
First, we quote a result of Shannon that allows us to restrict our attention to simplicial arrangements.
Next, we establish a sufficient (and in fact necessary) condition for a hyperplane arrangement to be a Coxeter arrangement:
All rank-two
subarrangements are Coxeter arrangements.  This reduction holds true in any
dimension.
Finally, instead of working with the
decomposition of space into regions, we  consider the
induced decomposition of the unit sphere.  In the three-dimensional
case, this
is a two-dimensional
spherical triangulation.  
The hyperplanes in the arrangement become great circles on the sphere and the rank-two subarrangements become collections of great circles all containing a common pair of antipodal points on the sphere.
We work explicitly with spherical geometry in order to prove that all these rank-two arrangements are dihedral.

In the final section we consider infinite affine arrangements in the
plane.  
In this case, there exist arrangements which are not reflection arrangements but which nonetheless have all regions isometric.

\section{Hyperplane arrangements}

In this section, we briefly review relevant definitions, point out that arrangements with isometric regions are necessarily simplicial, and reduce the problem of identifying Coxeter arrangements to a rank-two criterion.
All of the results of this section apply to arbitrary rank.

A \newword{finite real central hyperplane arrangement} is a finite collection $\A$ of linear hyperplanes (i.e.\ codimension-$1$ subspaces) in $\mathbb{R}^d$. 
An arrangement $\A$ in $\mathbb{R}^d$ is called \newword{essential} if the normal vectors to hyperplanes in $\A$ span $\mathbb{R}^d$.
Equivalently, $\A$ is essential if the intersection of its hyperplanes is
the origin.

A Coxeter group is a group arising from a certain kind of presentation by generators and relations.
Informally, the relations say that the generators act like linear reflections, and indeed, a finite group is a Coxeter group if and only if it has a representation as a (Euclidean) reflection group---a group generated by Euclidean reflections---in $\reals^d$.  
Given a specific representation of a Coxeter group $W$ as a reflection group, certain elements of $W$ act as reflections.
Each reflection has a \newword{reflecting hyperplane} or \newword{mirror}.
The \newword{Coxeter arrangement} associated to $W$ is the collection of the reflecting hyperplanes for all reflections in $W$.

It will be more convenient for us to work with an equivalent characterization of Coxeter arrangements as a \newword{closed system of mirrors}.
That is, a Coxeter arrangement is any hyperplane arrangement~$\A$ with the property that for any $H\in\A$, the (Euclidean) reflection fixing $H$ acts as a permutation of $\A$.

The \newword{regions} of an arrangement $\A$ are the connected components of $\mathbb{R}^d \setminus \A$.  
(Some authors use the term ``region'' for the closures of these connected components, but the distinction is unimportant here.)
The regions of a finite real central arrangement are unbounded convex polyhedra. 
If the arrangement is essential, then the regions are pointed polyhedral cones.

It is well known that the action of a Coxeter group on the regions of the corresponding Coxeter arrangement is simply transitive.  
Here, we only need transitivity, which is easy to see because each hyperplane is a mirror.
In particular, all the regions of a Coxeter arrangement are isometric.

In order to understand hyperplane arrangements in which all regions are isometric, we begin with an observation that allows us to only consider simplicial arrangements.
An arrangement is \newword{simplicial} if every region is the intersection of precisely $d$ open halfspaces.
Equivalently, each region is the positive linear span of $d$ linearly independent vectors. 
A simplicial arrangement is in particular essential.
The following theorem is due to Shannon~\cite{Shannon};
see also~\cite[Theorem~2.1.5]{OM}.

\begin{theorem}
\label{shannon}\cite{Shannon}
An essential arrangement $\mathcal{A}$ of $n$ hyperplanes in
$\mathbb{R}^d$ has at least $2n$ simplicial regions.
\end{theorem}

For our purposes, this theorem
is stronger than necessary. 
By the theorem, every essential arrangement has at least one simplicial region, so we have the following corollary.
\begin{corollary}\label{simp cor}
If $\A$ is essential and all of the regions of $\A$ are isometric, then $\A$ is simplicial.
\end{corollary}

\begin{remark}\label{nonessential}
In the proof of Theorem~\ref{thm:main}, we restrict our attention to essential arrangements.
We can do so because in dimensions $1$ and~$2$, the answer to Question~\ref{the q} is obviously no, and this negative answer lifts easily to non-essential arrangements in $\reals^3$.
Given a non-essential arrangement in $\reals^3$, one takes the quotient of $\reals^3$ modulo the intersection $\cap\A$ of all hyperplanes in $\A$, and the quotient of each hyperplane in $\A$ modulo $\cap\A$ to obtain a lower-dimensional arrangement.
For example if $\cap\A$ is a line $L$, then $\A$ is a ``pencil'' of planes containing~$L$, and the quotient is an arrangement of lines in a two-dimensional vector space.
The quotient arrangements is necessarily a Coxeter arrangement, and thus the original nonessential arrangement is also.
\end{remark}

Corollary~\ref{simp cor} and Remark~\ref{nonessential} let us restrict our attention to simplicial arrangements with isometric regions.
The following lemma gives us an easier criterion to check in order to conclude that an arrangement is a Coxeter arrangement, namely that codimension-$2$ suffices.
The converse is easy, but we do not need it.

\begin{definition}
Let $\A$ be a hyperplane arrangement in $\reals^d$.
A \newword{rank-two subarrangement} of $\A$ is a subset $\A'$ of $\A$ such that (i) there exists a codimension-$2$ subspace $U$ of $\mathbb{R}^d$ such that $\A'$ is the set of all hyperplanes in $\A$
containing $U$, and  (ii) $\A'$ has at least two hyperplanes.
\end{definition}

\begin{lemma}\label{reduction}
Let $\A$ be a finite central hyperplane arrangement in $\mathbb{R}^d$.  
If every rank-two subarrangement $\A'$ of $\A$ is a Coxeter arrangement then $\A$ is a Coxeter arrangement.
\end{lemma}
\begin{proof}
Let $\A$ be as above and assume every rank-two subarrangement is a
Coxeter arrangement.
We need to show that $\A$ is a closed system of mirrors.
That is, given hyperplanes $H_1$ and $H_2$ in~$\A$, writing $s_1$ for the the Euclidean reflection fixing $H_1$, we must show that the image of $H_2$ under~$s_1$ is also in $\A$. 
But $H_1$ and $H_2$ intersect in a linear subspace $U$ of dimension $d-2$, and the set of hyperplanes of $\A$ containing $U$ is a rank-two subarrangement $\A'$ of $\A$.
Since $\A'$ is a Coxeter arrangement, $s_1(H_2)$ is in $\A'$ and thus in $\A$ as desired.
\end{proof}

\section{The three-dimensional case}
\label{main}

In this section, we prove Theorem~\ref{thm:main} by rephrasing it as a statement in two-dimensional spherical geometry.
That is, instead of working directly with a simplicial arrangement of hyperplanes in $\reals^3$, we work with an arrangement of great circles on a sphere:
Each linear hyperplane intersects the unit sphere about the origin in a great circle.
Each region intersects the sphere in a spherical triangle.
The hypothesis that all regions in the arrangement are isometric is equivalent to the hypothesis that all of these triangles are congruent.
In light of Lemma~\ref{reduction}, to prove Theorem~\ref{thm:main}, it is enough to show that every rank-two subarrangement is a Coxeter arrangement.
Every rank-two subarrangement is the set of hyperplanes containing some line, corresponding to the set of all great circles containing some point (or in fact two antipodal points).
The hypothesis of Lemma~\ref{reduction} is equivalent to the statement that all of the angles at a given vertex are congruent.
Thus we can prove Theorem~\ref{thm:main} by establishing the following result.

\begin{theorem}\label{thm:main redux}
In any arrangement of great circles that cuts the sphere into congruent triangles, for any vertex, all angles incident to the vertex must be the same. 
\end{theorem}

Our argument for Theorem~\ref{thm:main redux} uses some well-known facts about spherical geometry.
First, consider a single spherical triangle with angles $\alpha$, $\beta$ and $\gamma$.  
The area of the triangle is $\al + \be + \ga - \pi$.
The fact that area is positive gives us the inequality $\al + \be + \ga > \pi$, which will be useful throughout the proof.
Also implicit throughout the proof is
the fact that a spherical triangle has exactly as many distinct side lengths as it has distinct angles.
This follows from the spherical law of sines:
If $a$, $b$, and $c$ are the side-lengths opposite the angles $\al$, $\be$, and $\ga$ respectively, then 
$\frac{\sin(a)}{\sin(\al)} = \frac{\sin(b)}{\sin(\be)} = \frac{\sin(c)}{\sin(\ga)}.$

Next, we need an easy lemma about the possible sequences of angles around a given vertex.  
We state the lemma for spherical geometry, but it is not special to the sphere.

\begin{lemma}[Parity Lemma]\label{parity}
Consider a triangulation of the sphere into isometric triangles with
three distinct angles.  
Then, when reading the angles around a vertex $v$,
if a maximal run of an angle has even length then
the two angles bordering the run are the same.
If a maximal run of an angle has odd length then
the two angles bordering the run are different.
\label{lemma_parity}
\end{lemma}

That is, if the angles are $\al \neq \be \neq \ga \neq \al$, a maximal run of $\be$'s must be bordered by the other angles in one of
the following
four ways:
\[\alpha,\!\!\underbrace{\,\beta, \ldots, \beta}_{\text{even number}}\!, \alpha,
\:\:\:\: \:\:\:\:
\gamma,\!\!\underbrace{\,\beta, \ldots, \beta}_{\text{even number}}\!, \gamma, 
\:\:\:\: \:\:\:\:
\alpha,\!\!\underbrace{\,\beta, \ldots, \beta}_{\text{odd number}}\!, \gamma 
\:\:\:\: \text{ or } \:\:\:\:
\gamma,\!\!\underbrace{\,\beta, \ldots, \beta}_{\text{odd number}}\!, \alpha .
\]

\begin{proof}
Assume that the maximal run consists of $k$ triangles with
the angle $\beta$ and, without loss of generality, that the angle $\alpha$ is before the run.  
Figure~\ref{fig:parity} shows the case $k=3$.
The first shared edge (the edge between the first and
second triangles) has length $c$. 
The next shared edge has length $a$, and the shared edges alternate between $a$ and $c$.
The result follows.
\end{proof}

\begin{figure}[h]
\includegraphics[height=1.5in]{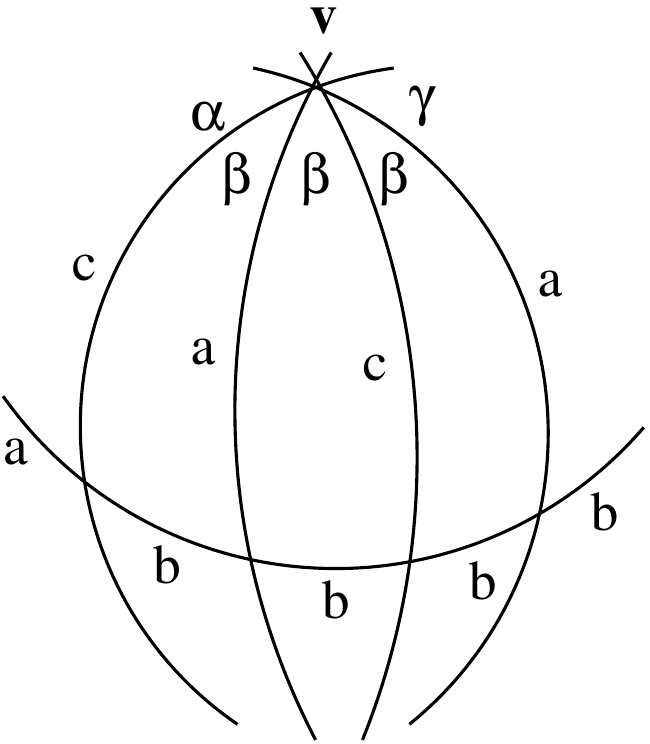}
\caption{An odd run of $\be$'s following the angle $\al$ must be followed by $\ga$.}
\label{fig:parity}
\end{figure}

We are now ready to prove
Theorem~\ref{thm:main redux},
which in turn proves the main
Theorem~\ref{thm:main}.

\begin{proof}[Proof of Theorem~\ref{thm:main redux}] 
Since all of the triangles defined by the arrangement of great circles are congruent, in particular, they have the same three angles.
Write $\al$, $\be$ and $\ga$ for these angles.
We need to show that around any vertex, at most one angle can appear.
We break up the argument based on the number of distinct angles appearing in $\set{\al,\be,\ga}$.

The first case is that the three angles are all distinct.
Consider the angles around a fixed vertex~$v$.  
Suppose all three angles appeared around $v$.  
At a minimum we would have six triangles meeting at~$v$ with each angle appearing exactly twice.  
The sum of the angles around $v$ would then be at least
$2(\al + \be + \ga)$.  
But then since $\al + \be + \ga>\pi$, the sum around $v$ would
be strictly larger than~$2\pi$, a contradiction.
Hence at most two distinct angles can appear around a fixed vertex.

We now claim that at least one of the angles is a right angle.
A well-known consequence of Euler's formula~\cite[Exercise~6.1.8]{West} is that there must exist a vertex of degree less than or equal to~$5$.
Since the triangles are defined by great circles,
each vertex is incident with an even number of angles.
Hence there is a vertex incident with exactly four angles.
Opposite angles at that vertex are the same, and the Parity Lemma (Lemma~\ref{parity}) implies that all four angles are the same.
Therefore all four angles are right angles, and the claim is proved.

Suppose two angles,
without loss of generality $\al$ and $\be$, appear around a vertex $v$.  
By the Parity Lemma, each run of a fixed angle must be of even length.
Since opposite angles across the vertex are the same, the runs come in opposite pairs across the vertex.  
Summing the angles around $v$ thus yields $2 \cdot 2m \al + 2 \cdot 2n \be = 2\pi$ for some positive integers $m$ and $n$.  Thus $\al + \be \leq m \al + n \be = \frac{\pi}{2}$.  
In particular, neither $\al$ nor $\be$ is $\pi/2$, so the claim says that $\ga=\frac\pi2$.
This yields the contradiction $\al + \be + \ga \leq \pi$.  
Hence we conclude that if all three angles of $\Delta$ are distinct, then no two different angles can meet at the same vertex, and we have finished the first case.

The second case is when two of the three angles are equal and distinct from the third.
Assume that the two distinct angles are $\al$ and $\be$ with $\be$ occurring twice and that the side opposite the angle~$\al$ has length $a$ while each side opposite an angle $\be$ has length $b$.

Consider the angles at a vertex $v$.
If $v$ is incident to an edge of length $a$, then this edge is shared by two triangles, both having the angle $\be$ at $v$.
Conversely, every angle $\be$ at $v$ includes a side of length $a$.
Thus if $v$ has degree greater than $4$, then the number of appearances of $\be$ incident to~$v$ is twice the number of edges of length $a$ incident to $v$.
We conclude that if $\al$ and $\be$ are both incident to $v$, then since opposite angles across $v$ are the same, there are at least $2$ appearances of~$\al$ and $4$ appearances of $\be$.
In particular $2 \al+4\be\le2\pi.$
Dividing by $2$, we contradict the inequality $\al+\be+\be>\pi$.
This contradiction finishes the second case.

In the third case, where $\al = \be = \ga$, the conclusion of the theorem is immediate.
\end{proof}

\section{Affine arrangements}

There are affine line arrangements
(with an infinite number of lines)
of the plane $\mathbb{R}^{2}$
which have isometric regions but are not
Euclidean reflection arrangements.
\begin{example}
{\rm
Take the affine arrangement $\widetilde{A}_{2}$
and apply a linear transformation, such that
the image of a region is
a triangle with at least two angles
different, that is, 
not an equilateral triangle.
The result is an arrangement where the angles around any vertex
are $\al$, $\be$, $\ga$, $\al$, $\be$ and $\ga$.
Observe that this agrees with the parity lemma,
which holds in the affine case by the same proof.}
\label{example_affine_one}
\end{example}
\begin{example}
{\rm
Take the affine arrangement $\widetilde{B}_{2}$,
that is,
all the lines of the form $x=k$, $y=k$, and $x\pm y=2k$ for $k\in\integers$.
This arrangement cuts the plane into congruent $(\pi/4,\pi/4,\pi/2)$-triangles.
Now apply the linear transformation
${\scriptsize \begin{pmatrix} c & 0 \\ 0 & 1 \end{pmatrix}}$
where $c$ is a positive number different from~$1$,
that is, a scaling in the $x$ direction.
The resulting arrangement cuts the plane into congruent $(\alpha,\beta,\pi/2)$-triangles.
Since $c \neq 1$, the
angles $\alpha$ and $\beta$ differ.
Note that this is not a reflection arrangement
since the triangle with vertices
$(0,0), (c,0), (c,1)$
is not a reflection of the triangle
$(0,0), (0,1), (c,1)$.}
\label{example_affine_two}
\end{example}
\begin{example}
{\rm
Start with the arrangement consisting of all integer
translates of the coordinate axis. Apply a linear transformation
such that the image of a region is a parallelogram, not a rectangle.
Here the angles around a vertex are
$\al$, $\be$, $\al$ and $\be$.
Note that this does not contradict the parity lemma
since the regions are quadrilaterals.}
\label{example_affine_three}
\end{example}

These three examples are the only families of affine arrangements that are not reflection arrangements and yet cut the plane into isometric regions (However, each arrangement is the image of a reflection arrangement under a linear map.) 
The fact that these are the only examples can be justified using similar techniques as in the proof of Theorem~\ref{thm:main redux}.
We omit the details.

\section*{Acknowledgments}
The first author was partially supported by
National Security Agency grant~H98230-13-1-0280
and wishes to thank the Mathematics Department of
Princeton University where he spent his sabbatical
while this paper was written.
The third author was supported by
National Science Foundation grant DMS-1101568.

\end{document}